\documentclass[a4paper]{article}

\usepackage[english]{babel}
\usepackage[utf8x]{inputenc}
\usepackage[T1]{fontenc}
\newcommand{\nm}{\noalign{\smallskip}}
\newcommand{\ds}{\displaystyle}
\usepackage{graphicx}
\usepackage{subcaption}
\usepackage{tikz}
\usepackage{pgfplots}
\usepackage{epstopdf}
\usepackage{wrapfig}
\usepackage{cutwin}
\graphicspath{ {MATLAB/main/} }
\usetikzlibrary{intersections,decorations.pathreplacing,decorations.markings,calc}
\usepackage{cite}
\usepackage{amsmath}
\usepackage{amsthm}
\usepackage{amssymb}
\numberwithin{equation}{section}
\newtheorem{prop}{Proposition}
\newtheorem{thm}{Theorem}
\newtheorem{lem}{Lemma}

\theoremstyle{definition}
\newtheorem{rmk}{Remark}
\usepackage{tikz}
\usepackage{graphicx}
\usepackage{hyperref}
\usepackage{caption}
\usepackage{float}
\theoremstyle{definition}
\newcommand{\D}{\mathcal{D}}
\newcommand{\Z}{\mathbb{Z}}
\newcommand{\R}{\mathbb{R}}
\newcommand{\C}{\mathcal{C}}
\newcommand{\A}{\mathcal{A}}
\newcommand{\B}{\mathcal{B}}

\newcommand{\K}{\mathcal{K}}
\newcommand{\Rc}{\mathcal{R}}
\newcommand{\p}{\partial}
\newcommand{\ie}{\textit{i.e.}}
\renewcommand{\L}{\mathcal{L}}
\renewcommand{\S}{\mathcal{S}}
\renewcommand{\H}{\mathcal{H}}
\renewcommand{\O}{\mathcal{O}}
\renewcommand{\epsilon}{\varepsilon}
\newcommand{\dx}{\: \mathrm{d}}

\newcommand{\eqnref}[1]{(\ref {#1})}
\def\nm{\noalign{\medskip}}

\title{Subwavelength resonances of encapsulated bubbles\thanks{\footnotesize The work of Hyundae Lee was supported by National Research Fund of Korea (NRF-2015R1D1A1A01059357, NRF-2018R1D1A1B07042678).}}

\author{ Habib Ammari\thanks{\footnotesize Department of Mathematics, 
ETH Z\"urich, 
R\"amistrasse 101, CH-8092 Z\"urich, Switzerland (habib.ammari@math.ethz.ch, brian.fitzpatrick@sam.math.ethz.ch, erik.orvehed.hiltunen@sam.math.ethz.ch sanghyeon.yu@sam.math.ethz.ch).} \and Brian Fitzpatrick\footnotemark[2] \and  Hyundae Lee\thanks{\footnotesize  Department of Mathematics, Inha University,  253 Yonghyun-dong Nam-gu,  Incheon 402-751,  Korea (hdlee@inha.ac.kr).}  \and Erik Orvehed Hiltunen\footnotemark[2]\and Sanghyeon Yu\footnotemark[2]}
\date{}

\begin{document}
	\maketitle

\begin{abstract}
The aim of this paper is to derive a formula for the subwavelength resonance frequency of an encapsulated bubble with arbitrary shape in two dimensions. Using Gohberg-Sigal theory, we derive an asymptotic formula for this resonance frequency, as a perturbation away from the resonance of the uncoated bubble, in terms of the thickness of the coating.
The formula is numerically verified in the case of circular bubbles, where the resonance can be efficiently computed using the multipole method.

\end{abstract}

\def\keywords2{\vspace{.5em}{\textbf{  Mathematics Subject Classification
(MSC2000).}~\,\relax}}
\def\endkeywords2{\par}
\keywords2{35R30, 35C20.}

\def\keywords{\vspace{.5em}{\textbf{ Keywords.}~\,\relax}}
\def\endkeywords{\par}
\keywords{bubble, subwavelength resonance, encapsulated bubble, thin coating.}

\section{Introduction}	
A gas bubble in a liquid is an acoustic scatterer which possesses a subwavelength resonance  called the Minnaert resonance \cite{first, Minnaert}. A remarkable feature of this resonance is its subwavelength scale; the size of the bubble can be several orders of magnitude smaller than the wavelength at the resonant frequency. This is due to the high contrast in density between the bubble and the surrounding medium and it opens up a wide range of applications, some examples being the creation of subwavelength phononic crystals \cite{DesignOfBandgap} or to achieve super-resolution in medical ultrasound-imaging \cite{ultrafastultrasound}.

Despite having interesting properties with regards to the creation of subwavelength scale metamaterials \cite{metasurface, superfocusing, effectivemedium, defect, bandgap}, bubbly media comprised of air bubbles inside water is highly unstable. There exist various approaches to stabilizing such  structures. One approach is to replace the background medium, water, with a soft elastic matrix, and it has been demonstrated that this technique results in metamaterials having properties similar to those of metamaterials comprised of air bubbles in water \cite{DesignOfBandgap, DesignOfBubbles}. Another approach is to encapsulate the bubbles in a thin coating \cite{reviewDoinikov, reviewFaez},  the aim being to prevent the fast dissolution and coalescence of the bubbles. Encapsulated bubbles have long been used  as ultrasound contrast agents, whereby the gas is trapped inside a coating of an albumin, polymer or lipid. However, the effect of such coating on the acoustic properties of the bubbly media has not yet been fully described.

Clearly, the introduction of a coating will affect the resonance frequency of the bubble, with the thin coating causing a slight perturbation of the Minnaert resonance. Through the application of layer potential techniques, asymptotic analysis and Gohberg-Sigal theory, we derive an original formula for the subwavelength resonance of an encapsulated bubble in two dimensions. Our results are complemented by several numerical examples which serve to validate them.

The paper is organized as follows. In Section \ref{sec-prelim}, we introduce some basic results regarding layer potentials and review the subwavelength resonance of an uncoated bubble in two dimensions. We also provide a correction to the formula for the Minnaert resonance in two dimensions, given in \cite{first}. In Section \ref{sec-problem}, we state the resonance problem for the encapsulated bubble. In Section \ref{sec:analysis}, we perform an asymptotic analysis in terms of the thickness of the coating, and use this to derive the resonance frequency in terms of the Minnaert frequency of the uncoated bubble. The main result is stated in Theorem \ref{thm:main} and equation \eqnref{eq:thm}. In Section \ref{sec-numerics}, we perform numerical simulations  to illustrate the main findings of this paper.  We make use of the multipole expansion method to validate our asymptotic formula for the subwavelength resonance of the encapsulated bubble in terms of its thickness. The paper ends with some concluding remarks. 

\section{Preliminaries} \label{sec-prelim}
In this section we state some well-known results about layer potentials. We also provide a correction to the Minnaert resonance formula in two dimensions given in \cite{first}.
\subsection{Layer potentials} \label{sec:layerpot}

For $k>0$ and $k=0$, let $\Gamma^k$ be the fundamental solution of the Helmholtz and  Laplace equations in dimension two, respectively, \ie{},
\begin{equation*}
\begin{cases}
\ds \Gamma^k(x,y) = -\frac{i}{4}H_0^{(1)}(k|x-y|), \ & k>0, \\
\nm
\ds \Gamma^0(x,y) = \frac{1}{2\pi}\ln|x-y|, & k=0,
\end{cases}
\end{equation*}
where $H_0^{(1)}$ is the Hankel function of the first kind of order zero. In the following, we will omit the superscript and denote this function by $H_0$.

Let $\S_{D}^k: L^2(\partial D) \rightarrow H_{\textrm{loc}}^1(\R^2)$ be the single layer potential defined by
\begin{equation*}
\S_D^k[\phi](x) = \int_{\partial D} \Gamma^k(x,y)\phi(y) \dx \sigma(y), \quad x \in \R^2.
\end{equation*}
Furthermore, let $\D_{D}^k: L^2(\partial D) \rightarrow H_{\textrm{loc}}^1(\R^2\setminus \partial D)$ be the double layer potential defined by
\begin{equation*}
\left(\D_D^k\right)[\phi](x) = \int_{\partial D} \frac{\partial }{\partial \nu_y}\Gamma^k(x,y) \phi(y) \dx \sigma(y), \quad x \in \R^2\setminus \partial D.
\end{equation*}
We also define the Neumann-Poincar\'e operator $\left(\K_D^k\right)^*: L^2(\partial D) \rightarrow L^2(\partial D)$ by
\begin{equation*}
\left(\K_D^k\right)^*[\phi](x) = \int_{\partial D} \frac{\partial }{\partial \nu_x}\Gamma^k(x,y) \phi(y) \dx \sigma(y), \quad x \in \partial D.
\end{equation*}
In the case when $k=0$, we will omit the superscripts and write $\S_D$, $\D_D$ and $\K_D^*$, respectively. The following so-called jump relations of $\S_D^k$ and $\D_D^k$ on the boundary $\partial D$ are well-known  (see, for instance, \cite{MaCMiPaP}):
\begin{equation*}
\S_D^k[\phi]\big|_+ = \S_D^k[\phi]\big|_-, \qquad \D_D^k[\phi]\big|_\pm = \left(\mp\frac{1}{2} I + \K_D^k\right) [\phi],
\end{equation*}
and
\begin{equation*}
\frac{\partial }{\partial \nu}\S_D^k[\phi] \bigg|_{\pm} =  \left(\pm\frac{1}{2} I + \left(\K_D^k\right)^*\right) [\phi], \qquad \frac{\partial }{\partial \nu}\D_D^k[\phi]\bigg|_+ = \frac{\partial }{\partial \nu}\D_D^k[\phi]\bigg|_-.
\end{equation*}
Here, $\partial/\partial \nu$ denotes the outward normal derivative,  and $|_\pm$ denotes the limits from outside and inside $D$. 

We now state some basic properties of the single-layer potential in two dimensions, given in \cite{MaCMiPaP}. The operator $-\frac{1}{2}I + \K_D^*$ is known to have a kernel of dimension 1. Let $\ker(-\frac{1}{2}I + \K_D^*) =\mathrm{span}(\psi_0)$ with $||\psi_0||=1$. Also, denote by $\phi_0 = \chi_{\partial D}$.  Then 
$$\S_D[\psi_0] = \gamma_0 \phi_0,$$
for some constant $\gamma_0$. It can be shown that $\S_D$ is invertible if and only if $\gamma_0 \neq 0$. In two dimensions, the fundamental solution of the free-space Helmholtz equation has a logarithmic singularity. Indeed, we have the following expansion \cite{MaCMiPaP}
\begin{equation}\label{eq:hankel}
-\frac{i}{4}H_0(k|x-y|) = \frac{1}{2\pi} \ln |x-y| + \eta_k + \sum_{j=1}^\infty\left( b_j \ln(k|x-y|) + c_j \right) (k|x-y| )^{2j},
\end{equation}
where $\ln$ is the principal branch of the logarithm and
$$ \eta_k = \frac{1}{2\pi}(\ln k+\gamma-\ln 2)-\frac{i}{4}, \quad b_j=\frac{(-1)^j}{2\pi}\frac{1}{2^{2j}(j!)^2}, \quad c_j=b_j\left( \gamma - \ln 2 - \frac{i\pi}{2} - \sum_{n=1}^j \frac{1}{n} \right),$$
and $\gamma$ is the Euler constant. Define 
\begin{equation*}
\hat{S}_D^k[\phi](x) = \S_D[\phi](x) + \eta_k\int_{\partial D} \phi\dx \sigma.
\end{equation*}
Then the following expansion holds:
\begin{equation} \label{eq:Sexpansion}
\S_D^k =  \hat{\S}_{D}^k +k^{2}\ln k \S_{D, 1}^{(1)} +k^{2} \S_{D, 1}^{(2)} + \O(k^4 \ln k),
\end{equation}
where
\begin{align*}
\mathcal{S}_{D, j}^{(1)} [\psi](x) &= \int_{\p D} b_j|x-y|^{2j} \psi(y)d\sigma(y),\\
\mathcal{S}_{D, j}^{(2)} [\psi](x) &= \int_{\p D} |x-y|^{2j}(b_j\ln|x-y|+c_j)\psi(y)d\sigma(y).
\end{align*}
Turning to the expansion of $\left(\K_D^k\right)^*$, we have
\begin{equation} \label{eq:Kexpansion}
\left(\K_{D}^k\right)^* = \K_{D}^* +k^{2}\ln k \K_{D, 1}^{(1)}+k^{2} \K_{D, 1}^{(2)} + \O(k^4 \ln k),
\end{equation}
where
\begin{align*}
\mathcal{K}_{D, j}^{(1)} [\psi](x) &= \int_{\p D} b_j\dfrac{\partial |x-y|^{2j}}{\partial \nu(x)}\psi(y)d\sigma(y),\\
\mathcal{K}_{D, j}^{(2)} [\psi](x) &= \int_{\p D} \dfrac{\partial \left( |x-y|^{2j}(b_j\ln|x-y|+c_j)\right)}{\nu(x)}\psi(y)d\sigma(y).
\end{align*}

The operator $\hat{S}_D^k$ is known to be invertible for any $k$ \cite{MaCMiPaP}. From this follows that $\S_D^k$ is invertible for $k$ small enough. For later reference, we conclude this section by defining the constant $a$ as 
$$a = \frac{\gamma_0+\langle \psi_0, \phi_0\rangle \eta_{k_b}}{\gamma_0+\langle \psi_0, \phi_0\rangle \eta_{k_w}}.$$

\subsection{Subwavelength resonance of a bubble}\label{subsec:bandgap}

Here, we briefly review the subwavelength resonance of a bubble as described in \cite{first}. 

Assume that the uncoated bubble occupies the bounded and simply connected domain $D$ with $\partial D \in C^{1,s}$ for some $0<s<1$. We denote by $\rho_b$ and $\kappa_b$ the density and the bulk modulus of the air inside the bubble, respectively. We let $\rho_w$ and $\kappa_w$ be the corresponding parameters for the water.  We introduce the variables
\begin{equation*} % \label{data1}
v_w = \sqrt{\frac{\kappa_w}{\rho_w}}, \quad v_b = \sqrt{\frac{\kappa_b}{\rho_b}}, \quad k_w= \frac{\omega}{v_w} \quad \text{and} \quad k_b= \frac{\omega}{v_b}
\end{equation*}
which represent the speed of sound outside and inside the bubble, and the wavenumber outside and inside the bubble, respectively. Also, $\omega$ means the operating frequency of acoustic waves.
We also introduce the dimensionless contrast parameter
\begin{equation*} % \label{data2}
\delta = \frac{\rho_b}{\rho_w}.
\end{equation*}
By choosing proper physical units, we may assume that the size of the bubble is of order one.  We assume  that the wave speeds outside and inside the bubbles are comparable to each other and that there is a large contrast in the density, that is, $$\delta \ll 1 \quad \text{and} \quad v_b, v_w = \O(1).$$

We consider the problem
\begin{equation} \label{eq-scattering-single}
\left\{
\begin{array} {ll}
&\ds \nabla \cdot \frac{1}{\rho_w} \nabla  u+ \frac{\omega^2}{\kappa_w} u  = 0 \quad \text{in} \quad \R^2 \backslash D, \\
\nm
&\ds \nabla \cdot \frac{1}{\rho_b} \nabla  u+ \frac{\omega^2}{\kappa_b} u  = 0 \quad \text{in} \quad D, \\
\nm	
&\ds  u|_{+} -u|_{-}  =0   \quad \text{on} \quad \partial D , \\
\nm
& \ds  \frac{1}{\rho_w} \frac{\partial u}{\partial \nu} \bigg|_{+} - \frac{1}{\rho_b} \frac{\partial u}{\partial \nu} \bigg|_{-} =0 \quad \text{on} \quad \partial D, \\
& u \ \text{satisfies the Sommerfeld radiation condition.}
\end{array}
\right.
\end{equation}
A resonance frequency to this problem is a complex number $\omega$ with negative imaginary part, such that a nonzero solution to equation \eqnref{eq-scattering-single} exists. In \cite{first}, it is proved that there exists a resonance frequency of subwavelength scale for this problem. The solution of \eqref{eq-scattering-single} has the following form:
\begin{equation} \label{Helm-solution}
u =
\begin{cases}
\mathcal{S}_{D}^{k_w} [\varphi]\quad & \text{in} ~ \mathbb{R}^2 \setminus \overline{D},\\
 \mathcal{S}_{D}^{k_b} [\psi]   &\text{in} ~   {D},
\end{cases}
\end{equation}
for some densities $\varphi, \psi \in  L^2(\p D)$. Using the jump relations for the single layer potentials, one can show that~\eqref{eq-scattering-single} is equivalent to the boundary integral equation
\begin{equation}  \label{eq-boundary}
M_0(\omega, \delta)[\Phi] =0,  
\end{equation}
where
\[
M_0(\omega, \delta) = 
 \begin{pmatrix}
  \mathcal{S}_D^{k_b} &  -\mathcal{S}_D^{k_w}  \\
  -\frac{1}{2}+ \left(\mathcal{K}_D^{k_b}\right)^*& -\delta\left( \frac{1}{2}+ \left(\mathcal{K}_D^{k_w}\right)^*\right)
\end{pmatrix}, 
\,\, \Phi = 
\begin{pmatrix}
\varphi \\
\psi
\end{pmatrix}.
\]

Since it can be shown that $\omega=0$ is a characteristic value for the operator-valued analytic function $M_0(\omega,0)$, we can conclude the following result by the Gohberg-Sigal theory \cite{MaCMiPaP, Gohberg1971}.
\begin{lem}\label{lem:GSchar}
For any $\delta$ sufficiently small, there exists a characteristic value 
$\omega_M= \omega_M(\delta)$ to the operator-valued analytic function 
$M_0(\omega, \delta)$
such that $\omega_M(0)=0$ and 
$\omega_M$ depends on $\delta$ continuously.
\end{lem}
In \cite{first}, an asymptotic formula for this characteristic value is computed. The formula in two dimensions is corrected in the following theorem.

\begin{thm} \label{thm:single}
	In the quasi-static regime, there exists resonances for a single bubble. Their leading order terms are given by the roots of the following equation:
	\begin{equation*}
	\omega^2 \ln \omega +
	\left[\left(1 + \frac{c_1}{b_1}-\ln v_b\right) + \frac{2\pi\gamma_0 }{(\psi_0, \phi_0)} \right] \omega^2 
	- \frac{v_b^2}{4Vol(D)} \frac{a \delta}{b_1} =0, 
	\end{equation*}
	where the constants $b_1, c_1, \gamma_0$ and $a$ are defined in Section \ref{sec:layerpot}.
\end{thm}
The root with positive real part is known as the \emph{Minnaert resonance} frequency, and will be denoted by $\omega_M = \omega_M(\delta)$.

\section{Encapsulated bubble: problem formulation} \label{sec-problem}
Consider now an encapsulated bubble, in which case $D$ is coated by a thin layer $D_l$ with a characteristic thickness $\epsilon$. Let $D_d = D_l \cup \overline{D}$ be the encapsulated bubble. We consider the following problem:
\begin{equation} \label{eq:scattering}
\left\{
\begin{array} {ll}
	&\ds \nabla \cdot \frac{1}{\rho_w} \nabla  u+ \frac{\omega^2}{\kappa_w} u  = 0 \quad \text{in} \quad \R^2 \backslash D_d, \\
	\nm
	&\ds \nabla \cdot \frac{1}{\rho_b} \nabla  u+ \frac{\omega^2}{\kappa_b} u  = 0 \quad \text{in} \quad D, \\
	\nm
	&\ds \nabla \cdot \frac{1}{\rho_l} \nabla  u+ \frac{\omega^2}{\kappa_l} u  = 0 \quad \text{in} \quad D_l, \\
	\nm	
	&\ds  u|_{+} -u|_{-}  =0   \quad \text{on} \quad \partial D \cup \partial D_d , \\
	\nm
	& \ds  \frac{1}{\rho_l} \frac{\partial u}{\partial \nu} \bigg|_{+} - \frac{1}{\rho_b} \frac{\partial u}{\partial \nu} \bigg|_{-} =0 \quad \text{on} \quad \partial D, \\
	& \ds  \frac{1}{\rho_w} \frac{\partial u}{\partial \nu} \bigg|_{+} - \frac{1}{\rho_l} \frac{\partial u}{\partial \nu} \bigg|_{-} =0 \quad \text{on} \quad \partial D_d,  \\
	& u \ \text{satisfies the Sommerfeld radiation condition.}
\end{array}
\right.
\end{equation}
Here, $\kappa_l$ and $\rho_l$ are the bulk modulus and density of the thin layer. Furthermore, define the two density contrast parameters $\delta_{bl}$ and $\delta_{lw}$ as 
$$\delta_{bl} = \frac{\rho_b}{\rho_l}, \qquad \delta_{lw} = \frac{\rho_l}{\rho_w}.$$
Observe that $\delta = \delta_{bl}\delta_{lw}$. We will consider the case when $\delta_{bl}$ is small while $\delta_{lw}$ is of order 1, and that all wave speeds are of order one, that is,

$$
\delta_{bl} \ll 1, \ \delta_{lw} = \O(1) \quad \text{and} \quad v_b, v_l, v_w = \O(1).
$$

In this paper, we want to show that, by encapsulating the  bubble $D$, there is a specific frequency  $\omega_\epsilon$ at which a non-trivial solution to the problem \eqnref{eq:scattering} exists. Moreover, we want to find an asymptotic formula for the frequency $\omega_\epsilon$ when $\epsilon$ is small.

\subsection{Integral representation of the solution}
We seek a solution $u(x)$ of the form
\begin{equation} \label{eq:sol}
u(x) = \begin{cases}
\S_{D}^{k_b}[\phi_1](x) \quad &x\in D, \\
\S_{D}^{k_l}[\phi_2](x) + S_{D_d}^{k_l} [\phi_3](x) & x\in D_l, \\
\S_{D_d}^{k_w}[\phi_4](x) & x \in \R^2\setminus \overline{D_d}.
\end{cases}
\end{equation}
A solution of this form satisfies the differential equation in \eqnref{eq:scattering}. Using the boundary conditions and the jump relations, it can be shown that the problem \eqnref{eq:scattering} admits a nonzero solution if and only if the layer densities $\phi_1,...,\phi_4$ are a nonzero solution to
\begin{equation} \label{eq:inteq}
\A(\omega, \epsilon, \delta)\Phi = 0,
\end{equation}
where $\Phi= \left(\begin{smallmatrix}\phi_1\\\phi_2 \\\phi_3\\\phi_4\end{smallmatrix}\right)$ and
\begin{equation*}
\A(\omega, \epsilon, \delta) = 
\begin{pmatrix}
\S_{D}^{k_b} &  -\S_{D}^{k_l} & -\S_{D,D_d}^{k_l} & 0 \\
0 & \S_{D_d,D}^{k_l} & \S_{D_d}^{k_l} & -\S^{k_w}_{D_d} \\
-\frac{1}{2}I+ \left(\K_{D}^{k_b}\right)^*& -\delta_{bl}\left( \frac{1}{2}I+ \left(\K_{D}^{k_l}\right)^*\right) & -\delta_{bl} \frac{\partial \S_{D,D_d}^{k_l}}{\partial \nu} & 0 \\
0 & \frac{\partial \S_{D_d,D}^{k_l}}{\partial \nu} & -\frac{1}{2}I+ \left(\K_{D_d}^{k_l}\right)^* & -\delta_{lw}\left( \frac{1}{2}I+ \left(\K_{D_d}^{k_w}\right)^*\right)
\end{pmatrix}.
\end{equation*}
Here the operator $\S_{D_d,D}^{k_w} = \S_{D}^{k_w}|_{x\in \partial D_d}$ is the restriction of $\S_{D}^{k_w}$ onto $\partial D_d$ and similarly for $\S_{D,D_d}^{k_w}$. Define $\H = L^2(\partial D) \times L^2(\partial D_d)\times L^2(\partial D) \times L^2(\partial D_d)$ and $\H_1 = H^1(\partial D) \times H^1(\partial D_d)\times L^2(\partial D) \times L^2(\partial D_d)$. It is clear that $\A$ is a bounded linear operator from $\H$ to $\H_1$, \ie{},  $\A \in \L(\H,\H_1)$.

\section{Asymptotic analysis} \label{sec:analysis}
In this section we expand the operator $\A(\omega, \epsilon, \delta)$ in terms of the small parameters $\epsilon,\omega$ and $\delta$. Using these expansions, we derive a formula for the perturbation $\omega_\epsilon-\omega_M$, which represents the shift of
the resonance of the encapsulated bubble $\omega_\epsilon$ away
from the resonance of the uncoated bubble, that
is, the Minnaert resonance frequency $\omega_M$. The key idea involves the use of a pole-pencil decomposition of the leading order term in the asymptotic expansion of $\A$ in terms of $\epsilon$, followed by the application of the generalized argument principle to find the characteristic value.

\subsection{Expansions as $\epsilon \rightarrow 0$} \label{sec:expansions}
Observe that the mapping $p: \partial D  \rightarrow \partial D_d,  \ p(x) = x+\epsilon  \nu_x$ is bijective. Let $x,y\in \partial D$ and let $\widetilde{x} = p(x) \in \partial D_d$ and $\widetilde{y} = p(y) \in \partial D_d$. Define $f: L^2( \partial D) \rightarrow L^2(\partial D_d), \ f(\phi)(\widetilde x) = \phi ( p^{-1}(\widetilde x)) $, and for a surface density $\phi$ on $\partial D$, define $\widetilde{\phi} = f(\phi)$ on $\partial D_d$. 

We define the signed curvature $\tau = \tau(x), x\in \partial D$ in the following way. Let $x = x(t)$ be a parametrization of $\partial D$ by arc length. Then define $\tau$ by
$$
\frac{d^2}{dt^2}x(t) = -\tau \nu_x.
$$
Recall that $\nu_x$ is defined as the \emph{outward} normal, and observe that $\tau$ is independent of the orientation of $\partial D$. The following proposition gives the expansion of the operators $\S^k_{D_d}, \S^k_{D,D_d}$ and $\S^k_{D_d,D}$ for small $\epsilon$.
\begin{prop}\label{prop:asympsingle}
Let $k>0$. 	Let $\phi \in L^2(\partial D)$ and let $x,y,\widetilde{x},\widetilde{y},\widetilde{\phi}$ be as above. Then 
\begin{equation} \label{eq:asympSdD}
\S_{D_d,D}^{k}[\phi](\widetilde{x}) = \S_D^k[\phi](x) +\epsilon \left(\frac{1}{2}I + \left(\K_D^k\right)^*\right)[\phi](x) + o(\epsilon),
\end{equation}
\begin{equation} \label{eq:asympSd}
\S_{D_d}^k[\widetilde{\phi}](\widetilde{x}) = \S_D^k[\phi](x) + \epsilon \left(\K_D^k + \left(\K_D^k\right)^*\right)[\phi](x) + \epsilon\S_D^k[\tau\phi](x) + o(\epsilon),
\end{equation}
\begin{equation} \label{eq:asympSDd}
\S_{D,D_d}^k[\widetilde{\phi}](x) = \S_D^k[\phi](x) + \epsilon \left(\frac{1}{2}I+ \K_D^k \right)[\phi](x) + \epsilon\S_D^k[\tau\phi](x) + o(\epsilon),
\end{equation}

\end{prop}
Here the $o(\epsilon)$ terms are in the pointwise $L^2$ sense, \ie{}, for any fixed $\phi$ we have 
\begin{equation*}
	\lim_{\epsilon \rightarrow 0} \frac{1}{\epsilon}\left\| \S_{D_d,D}^k[\phi](\widetilde{x}) -\left( \S_D^k[\phi](x) -\epsilon \left(-\frac{1}{2}I + \left(\K_D^k\right)^*\right)[\phi]\right) \right\|_{L^2(\partial D)} = 0,
\end{equation*}
and similarly for the other expansions.
\begin{proof}
The proof is given in \cite{asymptotics}, but in our case with the Taylor expansions in the $L^2$ sense (as given in \cite{weakdifffcn}, Theorem 3.4.2).
\end{proof}

\begin{prop} \label{prop:asympK}
	Let $\phi \in L^2(\partial D)$ and let $x,y,\widetilde{x},\widetilde{y},\widetilde{\phi}$ be as above. Then 
	\begin{equation} \label{eq:asympK}
	\left(\K_{D_d}^k\right)^*[\widetilde{\phi}](\widetilde{x}) = \left(\K_D^k\right)^*[\phi](x) + \epsilon\K_1^k[\phi](x) + o(\epsilon).
	\end{equation}
	Let $\tau$ be the curvature of $\partial D$. Then $\K_1^k$ is given by
	\begin{equation*}
	\K_1^k = \left(\K_D^k\right)^*[\tau\phi](x) - \tau(x) \left(\K_D^k\right)^*[\phi](x)  + \frac{\partial \D_D^k}{\partial \nu}[\phi](x) - \frac{\partial^2}{\partial T^2}\S_D^k[\phi](x) - k^2\S_D^k[\phi](x),
	\end{equation*}
	where $\frac{\partial^2}{\partial T^2}$ denotes the second tangential derivative, which is independent of the orientation of $\partial D$. 
\end{prop}
\begin{proof}
	The explicit expansion of $\K_{D_d}^*$ is derived in \cite{MaCMiPaP} for the Laplace case. We compute this in our case using similar arguments.  As derived in \cite{MaCMiPaP}, we have 
	\begin{equation} \label{eq:dsigma}
	\dx \sigma(\widetilde{y}) = \left( 1+\epsilon \tau(y) \right) \dx \sigma(y),
	\end{equation}
	Because the shapes of $\partial D$ and $\partial D_d$ are the same, we have $\widetilde \nu_{\widetilde x} = \nu_x$ for all $x\in \partial D$. Furthermore, we have
	\begin{equation}\label{eq:exp1}
	\frac{\partial}{\partial \widetilde{\nu}_{\widetilde{x}}} H_0(k|\widetilde{x}-\widetilde{y}|) = kH_0'(k|\widetilde{x}-\widetilde{y}|)\frac{\langle \widetilde{x}-\widetilde{y},\nu_x \rangle}{|\widetilde{x}-\widetilde{y}|}.
	\end{equation}
	We have 
	$$ |\widetilde{x}-\widetilde{y}|^2 = |x-y|^2 +2\epsilon \langle x-y, \nu_x-\nu_y\rangle + \epsilon^2|\nu_x-\nu_y|^2,$$
	and therefore the following expansions hold as $\epsilon \rightarrow 0$,
	$$ |\widetilde{x}-\widetilde{y}| = |x-y| +  \epsilon\frac{\langle x-y, \nu_x-\nu_y \rangle}{|x-y|} + \O(\epsilon^2),$$
	and
	$$ \frac{1}{|\widetilde{x}-\widetilde{y}|} = \frac{1}{|x-y|} -  \epsilon\frac{\langle x-y, \nu_x-\nu_y \rangle}{|x-y|^3} + \O(\epsilon^2).$$
	We therefore expand
	\begin{align} \label{eq:exp2}
	H_0'(k|\widetilde{x}-\widetilde{y}|) &= H_0'\left(k|x-y|\right) + \epsilon k H_0''\left(k|x-y|\right)\frac{\langle x-y,\nu_x-\nu_y\rangle}{|x-y|} + \O(\epsilon^2),
	\end{align} 
	and
	\begin{align} \label{eq:exp3}
	\frac{\langle \widetilde{x}-\widetilde{y},\nu_x \rangle}{|\widetilde{x}-\widetilde{y}|} = \frac{\langle {x}-{y},\nu_{{x}} \rangle}{|{x}-{y}|} + \epsilon \left( \frac{\langle \nu_x-\nu_y, \nu_x \rangle}{|x-y|}-\frac{\langle x-y,\nu_x-\nu_y \rangle \langle x-y,\nu_x\rangle }{|x-y|^3} \right) + \O(\epsilon^2).
	\end{align}
	Using the expansions \eqnref{eq:dsigma}, \eqnref{eq:exp1}, \eqnref{eq:exp2} and \eqnref{eq:exp3} we obtain
	\begin{align*}
	\left(\K_{D_d}^k\right)^*[\widetilde{\phi}](\widetilde{x}) &= -\frac{i}{4}\int_{D_d} \frac{\partial}{\partial \widetilde{\nu}_{\widetilde{x}}} H_0'(k|\widetilde{x}-\widetilde{y}|) \widetilde{\phi}({\widetilde{y}}) \dx\sigma(\widetilde{y}) \\
	&= -\frac{i}{4}\int_{\partial D} k H_0'\left(k|x-y|\right)\frac{\langle x-y,\nu_{x} \rangle}{|x-y|} \phi(y) \dx \sigma(y) \\
	& \quad + \epsilon \Bigg[-\frac{i}{4}\int_{\partial D} k H_0'\left(k|x-y|\right)\frac{\langle x-y,\nu_{x} \rangle}{|x-y|} \tau(y) \phi(y) \dx \sigma(y) \\
	& \qquad \qquad -\frac{i}{4}\int_{\partial D} k^2H_0''(k|x-y|) \frac{\langle x-y,\nu_x-\nu_y\rangle \langle x-y,\nu_x \rangle }{|x-y|^2}\phi(y)\dx \sigma(y)  \\
	& \qquad \qquad -\frac{i}{4}\int_{\partial D} k H_0'\left(k|x-y|\right)\left( \frac{\langle \nu_x-\nu_y,\nu_{x} \rangle}{|x-y|} -\frac{\langle x-y,\nu_x-\nu_y\rangle \langle x-y,\nu_x \rangle }{|x-y|^3} \right) \phi(y) \dx \sigma(y) \Bigg] \\
	& \quad + \O(\epsilon^2),
	\end{align*}
	 giving us the intermediate result
	\begin{multline} \label{eq:intermediate}
	\left(\K_{D_d}^k\right)^*[\widetilde{\phi}](\widetilde{x}) = \left(\K_d^k\right)^*[\phi](x) + \epsilon\Bigg[\left(\K_d^k\right)^*[\tau\phi](x) \\ 
	-\frac{i}{4}\int_{\partial D} k^2H_0''(k|x-y|) \frac{\langle x-y,\nu_x-\nu_y\rangle \langle x-y,\nu_x \rangle }{|x-y|^2}\phi(y)\dx \sigma(y) \\ 
	-\frac{i}{4}\int_{\partial D} k H_0'\left(k|x-y|\right)\left( \frac{\langle \nu_x-\nu_y,\nu_{x} \rangle}{|x-y|} -\frac{\langle x-y,\nu_x-\nu_y\rangle \langle x-y,\nu_x \rangle }{|x-y|^3} \right) \phi(y) \dx \sigma(y) \Bigg]+ \O(\epsilon^2).
	\end{multline}
	Observe that 
	\begin{align*}
	\frac{\partial \D_D^k}{\partial \nu}[\phi](x) = &-\frac{i}{4}\int_{\partial D} k^2H_0''(k|x-y|) \frac{\langle x-y,-\nu_y\rangle \langle x-y,\nu_x \rangle }{|x-y|^2}\phi(y)\dx \sigma(y)  \\
	&-\frac{i}{4}\int_{\partial D} k H_0'\left(k|x-y|\right)\left( \frac{\langle -\nu_y,\nu_{x} \rangle}{|x-y|} -\frac{\langle x-y,-\nu_y\rangle \langle x-y,\nu_x \rangle }{|x-y|^3} \right) \phi(y) \dx \sigma(y), \quad x\in \partial D,
	\end{align*}
	and that
	\begin{align*}
	\frac{\partial^2}{\partial T^2}\S_D^k[\phi](x) = &-\frac{i}{4}\int_{\partial D} k^2H_0''(k|x-y|) \frac{\langle x-y,T_x\rangle^2}{|x-y|^2}\phi(y)\dx \sigma(y)  \\
	&-\frac{i}{4}\int_{\partial D} k H_0'\left(k|x-y|\right)\left( \frac{1 - \tau(x)\langle x-y,\nu_{x} \rangle}{|x-y|} -\frac{\langle x-y,T_x\rangle^2}{|x-y|^3} \right)\phi(y) \dx \sigma(y), \quad x\in \partial D.
	\end{align*}
	Using these expressions in equation \eqnref{eq:intermediate}, together with the identity $|x-y|^2 = \langle x-y,\nu_x \rangle^2 + \langle x-y, T_x\rangle^2$, we obtain
	\begin{multline} \label{eq:almostthere}
	\left(\K_{D_d}^k\right)^*[\widetilde{\phi}](\widetilde{x}) = \left(\K_d^k\right)^*[\phi](x) + \epsilon\Bigg[\left(\K_d^k\right)^*[\tau\phi](x) - \tau(x) \left(\K_D^k\right)^*[\phi](x) + \frac{\partial \D_D^k}{\partial \nu}[\phi](x) - \frac{\partial^2}{\partial T^2}\S_D^k[\phi](x) \\ 
	-\frac{i}{4}\int_{\partial D} k^2H_0''(k|x-y|) \phi(y)\dx \sigma(y) 
	-\frac{i}{4}\int_{\partial D} \frac{k H_0'\left(k|x-y|\right)}{|x-y|} \phi(y) \dx \sigma(y) \Bigg]+ \O(\epsilon^2).
	\end{multline}
	Applying standard relations for Bessel functions, we have the relation 
	$$\frac{H_0'(k|x-y|)}{k|x-y|} + H_0''(k|x-y|) = -H_0(k|x-y|).$$
	Using this in equation \eqnref{eq:almostthere}, we find
	\begin{multline*}
	\left(\K_{D_d}^k\right)^*[\widetilde{\phi}](\widetilde{x}) = \left(\K_d^k\right)^*[\phi](x) + \epsilon\Bigg[\left(\K_d^k\right)^*[\tau\phi](x) - \tau(x) \left(\K_D^k\right)^*[\phi](x) + \frac{\partial \D_D^k}{\partial \nu}[\phi](x) \\
	- \frac{\partial^2}{\partial T^2}\S_D^k[\phi](x) - k^2\S_D^k[\phi](x) \Bigg]+ \O(\epsilon^2),
	\end{multline*}
	which is the desired result.
\end{proof}

\begin{prop} \label{prop:asympderiv}
	Let $\phi \in H^1(\partial D)$ and let $x,y,\widetilde{x},\widetilde{y},\widetilde{\phi}$ be as above. Then 
\begin{equation} \label{eq:asymppSdD}
\frac{\partial \S_{D_d,D}^k[\phi]}{\partial \widetilde{\nu}} (\widetilde{x}) = \left(\frac{1}{2}I + \left(\K_D^k\right)^*\right)[\phi](x) + \epsilon \mathcal{R}_D^k[\phi](x) + o(\epsilon),
\end{equation}
\begin{equation} \label{eq:asymppSDd}
\frac{\partial \S_{D,D_d}^k[\widetilde{\phi}]}{\partial \nu} (x) = \left(-\frac{1}{2}I + \left(\K_D^k\right)^*\right)[\phi](x) + \epsilon \L_D^k[\phi](x) + o(\epsilon),
\end{equation}
where $\mathcal{R}_D^k$ and $\L_D^k$ are given by
\begin{equation*}
\mathcal{R}_D^k[\phi](x) = -k^2\S_D^k[\phi](x) - \tau(x)\left(\frac{1}{2}I + \left(\K_D^k\right)^*\right)[\phi](x) -\frac{\partial^2}{\partial T^2}\S_D^k[\phi](x),
\end{equation*}
\begin{equation*}
\L_D^k[\phi](x) = \left(-\frac{1}{2}I + \left(\K_D^k\right)^*\right)[\tau\phi](x) +\frac{\partial \D_D^k}{\partial \nu}[\phi](x).
\end{equation*}
\end{prop}
\begin{proof}
The proof is similar to the one given in \cite{asymptoticsderivative}, but adjusted for the Helmholtz case. Because $\phi\in H^1(\partial D)$ we have that $\S_D^k[\phi]\in H^2(\partial D)$. Because the normals $\widetilde{\nu}_{\widetilde{x}}$ and $\nu_x$ coincide, we have
\begin{align*}
\frac{\partial \S_{D_d,D}^k[\phi]}{\partial \widetilde{\nu}} (\widetilde{x}) &= \nu \cdot \nabla \S_{D_d,D}^k[\phi](\widetilde{x}) \\
&= \frac{\partial \S_D^k[\phi]}{\partial \nu} \bigg|_{+}(x) + \epsilon\left(\frac{\partial^2}{\partial \nu^2}\S_D^k[\phi]\bigg|_{+}(x) \right) + o(\epsilon).
\end{align*}
Using the Laplacian in the curvilinear coordinates defined by $T_x,\nu_x$ for $x\in \partial D$,
\begin{equation*}\label{eq:lapcurve}
\Delta = \frac{\partial^2}{\partial \nu^2} + \tau\frac{\partial}{\partial \nu} + \frac{\partial^2}{\partial T^2},
\end{equation*}
we find
\begin{equation*}
\frac{\partial^2}{\partial \nu^2}\S_D^k[\phi]\bigg|_{+}(x) = -k^2\S_D^k[\phi](x) -\tau(x)\frac{\partial \S_D^k[\phi]}{\partial \nu}\bigg|_{+}(x) - \frac{\partial^2 \S_D^k[\phi]}{\partial T^2}(x),
\end{equation*}
so equation \eqnref{eq:asymppSdD} follows using the jump relations. To derive equation \eqnref{eq:asymppSDd}, pick a function $f\in H^1(\partial D)$. Then there exists a solution $u$ to the Dirichlet problem
\begin{equation*}
\begin{cases}
\Delta u + k^2 u = 0 \quad &\text{in } D ,\\
u = f & \text{in } \partial D.
\end{cases}
\end{equation*}
Using duality and integration by parts in the interior region, we obtain that
\begin{align*}
\int_{\partial D} \frac{\partial \S_{D,D_d}^k[\widetilde{\phi}]}{\partial \nu} (x) f(x) \dx \sigma(x) &=\int_{D}\left(u\Delta \S_{D,D_d}-\S_{D,D_d}\Delta u \right)\dx x +  \int_{\partial D} \S_{D,D_d}^k[\widetilde{\phi}](x)\frac{\partial f }{\partial \nu} (x)  \dx \sigma(x) \\
&=\int_{\partial D} \S_{D,D_d}^k[\widetilde{\phi}](x)\frac{\partial f }{\partial \nu} (x)  \dx \sigma(x) \\
&= \int_{\partial D_d} \S_{D_d,D}^k\left[\frac{\partial f }{\partial \nu}\right](\widetilde{x})\widetilde{\phi}(\widetilde{x}) \dx \sigma(\widetilde{x}).
\end{align*}
Combining Proposition \ref{prop:asympsingle} together with \eqnref{eq:dsigma}, we find
\begin{align*}
\int_{\partial D_d} \S_{D_d,D}^k\left[\frac{\partial f }{\partial \nu}\right](\widetilde{x})\widetilde{\phi}(\widetilde{x}) \dx \sigma(\widetilde{x}) &= \int_{\partial D}\left( \S_D^k\left[\frac{\partial f }{\partial \nu}\right] (x) +\epsilon \left(\frac{1}{2}I + \left(\K_D^k\right)^*\right)\left[\frac{\partial f }{\partial \nu}\right](x)\right)\phi(x)\left(1+\epsilon \tau(x)\right)\dx \sigma(x) + o(\epsilon) \\
&= \int_{\partial D}\S_D^k\left[\frac{\partial f }{\partial \nu}\right] \phi\dx \sigma +\epsilon \int_{\partial D}\left( \tau(x)\S_D^k + \frac{1}{2}I + \left(\K_D^k\right)^*\right)\left[\frac{\partial f }{\partial \nu}\right]\phi\dx \sigma + o(\epsilon) \\
& = \int_{\partial D} \S_D^k\left[\phi\right]\frac{\partial f }{\partial \nu} \dx \sigma +\epsilon \int_{\partial D}\left(\S_D^k[\tau\phi] +\left(\frac{1}{2}I + \K_D^k\right)\left[\phi\right]\right)\frac{\partial f }{\partial \nu}\dx \sigma + o(\epsilon) \\
&= \int_{\partial D} \frac{\partial \S_D^k }{\partial \nu}[\phi]\bigg|_{-}f \dx \sigma +\epsilon \int_{\partial D}\left(\frac{\partial \S_D^k }{\partial \nu}\bigg|_{-}[\tau\phi] +\frac{\partial \D_D^k}{\partial \nu} \left[\phi\right]\right)f\dx \sigma + o(\epsilon).
\end{align*}
Therefore, \eqnref{eq:asymppSDd} follows using the jump formulas.
\end{proof}

\subsection{Expansion of $\A$}
Observe that Proposition \ref{prop:asympderiv} assumes $\phi \in H^1(\partial D)$. Define by $\H_2 = H^1(\partial D) \times H^1(\partial D_d) \times H^1(\partial D) \times H^1(\partial D_d)$. We seek the solution to equation \eqnref{eq:inteq}, and it is clear that this solution satisfies $\Phi \in \H_2$. In the following, we will consider $\A$ as an operator on the space $\H_2$.

Define the bijection $F: \left(H^1(\partial D)\right)^4 \rightarrow \H_2, F = (id, f, id, f)$, where $f$ is defined as in Section \ref{sec:expansions}.  Using the asymptotic expansions \eqnref{eq:asympSdD}, \eqnref{eq:asympSd}, \eqnref{eq:asympSDd}, \eqnref{eq:asympK}, \eqnref{eq:asymppSdD} and \eqnref{eq:asymppSDd}, we can expand the operator $\A$ as 
\begin{equation*}
\A(\omega, \epsilon, \delta) = F \circ \left(\hat{\A}_0(\omega,\delta) + \epsilon \hat{\A}_1(\omega,\delta) + o(\epsilon)\right)\circ F^{-1},
\end{equation*}
where
\begin{equation*} \label{eq:A0}
\hat{\A}_0(\omega,\delta) = 
\begin{pmatrix}
\S_{D}^{k_b} &  -\S_{D}^{k_l} & -\S_{D}^{k_l} & 0 \\
0 & \S_{D}^{k_l} & \S_{D}^{k_l} & -\S_D^{k_w} \\
-\frac{1}{2}I+ (\K_{D}^{k_b})^*& -\delta_{bl}\left( \frac{1}{2}I+ (\K_{D}^{k_l})^*\right) & -\delta_{bl}\left( -\frac{1}{2}I+ (\K_{D}^{k_l})^*\right) & 0 \\
0 & \frac{1}{2}I+ (\K_D^{k_l})^* & -\frac{1}{2}I+ (\K_D^{k_l})^* & -\delta_{lw}\left( \frac{1}{2}I+ \left(\K_D^{k_w}\right)^*\right)
\end{pmatrix}, 
\end{equation*}
and
\begin{equation*} \label{eq:A1}
\hat{\A}_1(\omega,\delta) = 
\begin{pmatrix}
0 & 0 & -\left(\frac{1}{2}I+ \K_D^{k_l} +\S_D^{k_l}[\tau\cdot] \right)  & 0 \\
0 & \frac{1}{2}I + \left(\K_D^{k_l}\right)^*  & \K_D^{k_l} + \left(\K_D^{k_l}\right)^* + \S_D^{k_l}[\tau\cdot] & -\left(\K_D^{k_w} + \left(\K_D^{k_w}\right)^* + \S_D^{k_w}[\tau\cdot]\right) \\
0 & 0 & -\delta_{bl} \L_D^{k_l} & 0 \\
0 & \mathcal{R}_D^{k_l} & \K_1^{k_l} & -\delta_{lw}\K_1^{k_w}
\end{pmatrix}.
\end{equation*}
It is clear that $\omega_\epsilon$ is a characteristic value for $\A$ if and only if $\omega_\epsilon$ is a characteristic value for $F^{-1}\circ\A\circ F = \hat{\A}_0(\omega,\delta) + \epsilon \hat{\A}_1(\omega,\delta) + o(\epsilon)$. Using elementary row reductions, it is clear that this operator has the same characteristic values as
\begin{equation*}
\A_0 + \epsilon \A_1 + o(\epsilon),
\end{equation*}
where
\begin{equation*} \label{eq:Ahat0}
\A_0(\omega,\delta) = 
\begin{pmatrix}
\S_{D}^{k_b} &  0 & 0 & -\S_D^{k_w} \\
\S_{D}^{k_b} & -\S_{D}^{k_l} & -\S_{D}^{k_l} & 0 \\
0 & \frac{1}{2}I+ (\K_D^{k_l})^* & -\frac{1}{2}I+ (\K_D^{k_l})^* & -\delta_{lw}\left( \frac{1}{2}I+ \left(\K_D^{k_w}\right)^*\right) \\
-\frac{1}{2}I+ (\K_{D}^{k_b})^* & 0 & 0 & -\delta\left( \frac{1}{2}I+ \left(\K_D^{k_w}\right)^*\right)\\
\end{pmatrix}, 
\end{equation*}
and
\begin{equation*} \label{eq:Ahat1}
\A_1(\omega,\delta) = 
\begin{pmatrix}
0 & \frac{1}{2}I + \left(\K_D^{k_l}\right)^*  & -\frac{1}{2}I + \left(\K_D^{k_l}\right)^* & -\left(\K_D^{k_w} + \left(\K_D^{k_w}\right)^* + \S_D^{k_w}[\tau\cdot]\right) \\
0 & 0 & -\left(\frac{1}{2}I+ \K_D^{k_l} +\S_D^{k_l}[\tau\cdot] \right)  & 0 \\
0 & \mathcal{R}_D^{k_l} & \K_1^{k_l} & -\delta_{lw}\K_1^{k_w} \\
0 & \delta_{bl} \mathcal{R}_D^{k_l}  & \delta_{bl}\left( \tau(x)I + \mathcal{R}_D^{k_l} \right) & -\delta\K_1^{k_w} \\
\end{pmatrix}.
\end{equation*}
In these equations, recall that the three contrast parameters are related by $\delta = \delta_{bl}\delta_{lw}$. The following proposition is one of the key steps in computing the resonance frequency.
\begin{prop}\label{prop:polepencil}
	Let $\omega_M$ be the Minnaert resonance for the uncoated bubble. For $\omega$ in a punctured neighbourhood of $\omega_M$ and for $\delta$ small enough, $\A_0(\omega,\delta)$ is an injective operator. Furthermore, the following pole-pencil decomposition holds
	\begin{equation} \label{pencil}
	\big(\A_0(\omega,\delta)\big)^{-1} = \frac{L}{\omega-\omega_M} + R(\omega),
	\end{equation}
	where $R(\omega)$ is holomorphic, $L: \left(L^2(\partial D)\right)^4 \rightarrow \ker(\A_0(\omega_M,\delta))$ and $\dim\ker(\A_0(\omega_M,\delta)) = 1$.
\end{prop}

\begin{proof}
	The first and the fourth row of $\A_0$ decouples, which leads to the matrices
	\begin{equation*}
	M_0=\begin{pmatrix}
	\S_{D}^{k_b} & -\S_D^{k_w} \\
	-\frac{1}{2}I+ (\K_{D}^{k_b})^* & -\delta\left( \frac{1}{2}I+ \left(\K_D^{k_w}\right)^*\right)\\
	\end{pmatrix}
	\end{equation*}
	and
	\begin{equation*}
	M_d=\begin{pmatrix}
	-\S_{D}^{k_l} & -\S_{D}^{k_l} \\
	\frac{1}{2}I+ \left(\K_D^{k_l}\right)^* & -\frac{1}{2}I+ \left(\K_D^{k_l}\right)^*\\
	\end{pmatrix}.
	\end{equation*}
	$M_0$ is the operator which corresponds to the uncoated bubble, and is known to have a discrete set of characteristic values \cite{first}. Hence there is a punctured neighbourhood of $\omega_M$ where $M_0$ is invertible. It is easily shown that $M_d$ is invertible if and only if $\S_D^{k_w}$ is invertible. Because $\omega_M$ is of  subwavelength scale, and tends to zero as $\delta \rightarrow 0$, $\S_D^{k_w}$ is invertible for $\delta$ small enough. It follows that $\A_0$ is invertible for $\omega$ in a punctured neighbourhood of $\omega_M$. Because $\ker M_0(\omega_M,\delta)$ is one-dimensional \cite{first} and because $M_d$ is invertible, it follows that $\ker(\A_0(\omega_M,\delta))$ is one-dimensional. Finally, because $\omega_M$ is a pole of order one for $M_0$ \cite{first}, and because $M_d$ is invertible, it follows that $\omega_M$ is a pole of order one for $\A_0$.
\end{proof}
Using the expansion of $\A$, Lemma \ref{lem:GSchar} and the observation that $\A_0$ has the same characteristic values as $M_0$,  Gohberg-Sigal theory implies the following result.
\begin{lem} \label{lem:GScharA}
	For any $\epsilon$ and $\delta$ sufficiently small, there exists a characteristic value 
	$\omega_\epsilon= \omega_\epsilon(\epsilon,\delta)$ to the operator-valued analytic function 
	$\mathcal{A}(\epsilon,\omega, \delta)$
	such that $\omega_\epsilon(0,\delta) = \omega_M(\delta)$ and 
	$\omega_\epsilon$ depends continuously on $\epsilon$ and $\delta$.
\end{lem}

Since $\ker \A_0(\omega_M,\delta)$ is one-dimensional, define $\Psi$ and $\Phi$ by
\begin{align*}
\ker \A_0(\omega_M,\delta) &= \mathrm{span}(\Psi), \\
\ker \A_0^*(\omega_M,\delta) &= \mathrm{span}(\Phi).
\end{align*}
In the next sections, we compute $\Psi$, $\Phi$ and $L$.

\subsection{Computation of $\Psi$ and $\Phi$} \label{sec:psi}
We make use of the computations in \cite{first}, and asymptotically expand $\A_0(\omega,\delta)$ in terms of $\omega$ and $\delta$. We are interested in the case when the contrast parameter $\delta$ is small, while $\delta_{lw} = \O(1)$, \ie{} the contrast between the layer of coating and the water is of order one. Taking into consideration Lemma \ref{lem:GScharA},  we will assume $\omega$ is close to $\omega_M$, and Theorem \ref{thm:single} shows that this gives
\begin{equation} \label{eq:regime}
\omega^2\ln\omega = \O(\delta).
\end{equation}
Define $\A_0^0$ as
\begin{equation*}
\A_0^0 = 
\begin{pmatrix}
\hat\S_{D}^{k_b} &  0 & 0 & -\hat\S_D^{k_w} \\
\hat\S_{D}^{k_b} & -\hat\S_{D}^{k_l} & -\hat\S_{D}^{k_l} & 0 \\
0 & \frac{1}{2}I+ (\K_D)^* & -\frac{1}{2}I+ (\K_D)^* & -\delta_{lw}\left( \frac{1}{2}I+ \left(\K_D\right)^*\right) \\
-\frac{1}{2}I+ (\K_{D})^* & 0 & 0 & 0
\end{pmatrix}.
\end{equation*}
In light of the expansion \eqnref{eq:Sexpansion}, we have $\A_0(\omega,\delta) = \A_0^0 + \B(\omega,\delta)$ with $\B(\omega,\delta) = \O(\delta)$. Let $\Psi_0$ be such that $\mathrm{span}\{\Psi_0\}= \ker (\A_0^0)$. Then $\Psi(\omega,\delta) =  \Psi_0 + \O(\delta)$. Let us write  $$\Psi_0 = \alpha_0 \left(\begin{smallmatrix}\psi_1 \\\psi_2 \\\psi_3 \\\psi_4 \end{smallmatrix}\right),$$
where $\alpha_0$ is a normalization constant. Observe that the equation $\A_0^0\Psi_0=0$ is equivalent to 
\begin{equation*}
\begin{pmatrix}
\hat\S_{D}^{k_b} & -\hat\S_D^{k_w} \\
-\frac{1}{2}I+ \K_{D}^* & 0\\
\end{pmatrix}
\begin{pmatrix}
\psi_1 \\ \psi_4
\end{pmatrix}
=0 \ \  \text{and} \ 
\begin{pmatrix}
\hat\S_{D}^{k_l} & \hat\S_{D}^{k_l} \\
\frac{1}{2}I+ \K_D^* & -\frac{1}{2}I+ \K_D^*
\end{pmatrix}
\begin{pmatrix} 
\psi_2 \\
\psi_3
\end{pmatrix}
=
\begin{pmatrix}
\hat\S_D^{k_b}[\psi_1] \\
\delta_{lw}\left(\frac{1}{2}+ \K_D^*\right)[\psi_4]
\end{pmatrix}.
\end{equation*}
As before, let $\phi_0 = \chi_{\partial D}$ and let $\psi_0$ be a solution to
\begin{equation*}
\left(-\frac{1}{2}I+ \K_{D}^*\right)\psi_0=0, \ \ \int_{\partial D} |\psi_0|^2 \dx \sigma = 1.
\end{equation*}
Observe that $\psi_0$ is unique up to sign. Define by $c := \langle \psi_0, \phi_0 \rangle$. Clearly, we can choose $\psi_1 = \psi_0$. In \cite{first}, it is shown that $\psi_4 = a\psi_0$, where 
$$a = \frac{\gamma_0+  c \eta_{k_b}}{\gamma_0+ c \eta_{k_w}},$$
as defined in Section \ref{sec:layerpot}.
Defining 
$$a_l = \frac{\gamma_0+ c \eta_{k_b}}{\gamma_0+ c \eta_{k_l}},$$
it is easily shown that 
\begin{equation}\label{eq:psi1234}
\Psi_0 = \alpha_0 \begin{pmatrix}
\psi_0 \\
\delta_{lw}a\psi_0\\
(a_l-a\delta_{lw})\psi_0\\
a\psi_0
\end{pmatrix}
\end{equation}
We now turn to $\Phi$. Define $\Phi_0$ by $\left(\A_0^0\right)^* \Phi_0 = 0$. Then $\Phi = \Phi_0 + \Phi_1 $, where $\Phi_1 = \O(\delta)$. A direct computation gives
$$
\Phi_0 = \beta_0
\begin{pmatrix}
0 \\ 0 \\ 0 \\ \phi_0
\end{pmatrix},
$$
where $\beta_0$ is a normalization constant. We will also need the leading order term of $\Phi_1$. It is easily seen that $\Phi_1$ has the form
$$\Phi_1 = 
\begin{pmatrix}
u_1 \\ 0 \\ 0 \\ u_4
\end{pmatrix},
$$
where $u =  \left(\begin{smallmatrix} u_1 \\ u_4 \end{smallmatrix} \right)$. Using the methods from \cite{first}, it is easily shown that $u$ is given by 
\begin{equation} \label{eq:u}
u = -\beta_0\left(\tilde{M}_0^*\right)^{-1}\B_0^* \begin{pmatrix} 0\\ \phi_0 \end{pmatrix}.
\end{equation}
Here, $\B_0$ is the $2\times 2$ matrix given as the first and fourth rows and columns of $\B$, and $\tilde{M}_0$ is given by
$$
\tilde{M}_0^* = M_0^* + \big\langle \cdot , \left(\begin{smallmatrix} 0\\ \phi_0 \end{smallmatrix}\right)
\big\rangle \left(\begin{smallmatrix} \psi_0\\ a\psi_0 \end{smallmatrix}\right).
$$
From the expansions given in \eqnref{eq:Sexpansion} and \eqnref{eq:Kexpansion}, we have
$$
\B_0^* = 
\bar{\omega}_M^2\ln\bar{\omega}_M\begin{pmatrix}
v_b^{-2}\S_{D,1}^{(1)} & v_b^{-2}\left(K_{D,1}^{(1)}\right)^* \\
-v_w^{-2}\S_{D,1}^{(1)} & 0
\end{pmatrix} 
+ \delta \begin{pmatrix}
0 & 0 \\
0 & -\left(\frac{1}{2}+\K_D\right)
\end{pmatrix} + \O(\omega^2).
$$
In \cite{first}, it is shown that $\left(\K_{D,1}^{(1)}\right)^*[\phi_0] = 4b_1Vol(D)\phi_0$. It follows that 
$$
\B_0^*\begin{pmatrix}
0 \\
\phi_0
\end{pmatrix}
=
\frac{4b_1Vol(D)}{v_b^2}\bar{\omega}_M^2\ln\bar{\omega}_M \begin{pmatrix}
\phi_0 \\
0
\end{pmatrix} - \delta\begin{pmatrix}
0 \\
\phi_0
\end{pmatrix} + \O(\omega^2).
$$
Now, from \eqnref{eq:u} we know that the leading order term $u^{(1)}$ of $u$ is the solution to the equation
$$
\tilde{M}_0^*u^{(1)} = -\beta_0\frac{4b_1Vol(D)}{v_b^2}\bar{\omega}_M^2\ln\bar{\omega}_M \begin{pmatrix}
\phi_0 \\
0
\end{pmatrix} + \delta\begin{pmatrix}
0 \\
\phi_0
\end{pmatrix}.
$$
Observe that 
\begin{align*}
\left(\hat{\S}_D^{k}\right)^*[\psi_0] &= \S_D[\phi_0] + \bar{\eta}_k\int_{\partial D}\psi_0 \dx \sigma \\
&= \left(\gamma_0 + \bar{\eta}_kc\right)\phi_0.
\end{align*}
In this equation, recall that $c = \langle \psi_0, \phi_0 \rangle$. Assume now that $u^{(1)}$ has the form $u^{(1)} = \left(\begin{smallmatrix} y_1\psi_0\\ 0 \end{smallmatrix}\right)$. Then 
$$
\tilde{M}_0^*u^{(1)} = M_0^*u^{(1)} = y_1\begin{pmatrix}
-\left(\gamma_0 + \bar{\eta}_{k_b}c\right) \phi_0\\
\left(\gamma_0 + \bar{\eta}_{k_w}c\right) \phi_0
\end{pmatrix}.
$$
Finally, by Theorem \ref{thm:single} we have 
$$
\delta = \frac{4b_1Vol(D)}{a_wv_b^2}\omega_M^2\ln\omega_M + \O(\omega^2),
$$
which shows that $y_1 = -\beta_0\frac{\delta}{\left(\gamma_0 + \bar{\eta}_{k_w}c\right)}$, and that 
$$\Phi_1 = -\beta_0\frac{\delta}{\left(\gamma_0 + \bar{\eta}_{k_w}c\right)}
\begin{pmatrix}
\psi_0 \\ 0 \\ 0 \\ 0
\end{pmatrix} + \O(\omega^2).
$$

\subsection{Computation of $L$}
Let $L$ be defined by (\ref{pencil}).
From \eqnref{pencil} we obtain
$$
L \A_0(\omega_M)= 0 \quad \mbox{ and } \quad \A_0(\omega_M) L = 0.  
$$
Therefore, $L$ maps $L^2(\partial D)$ into $\ker (\A_0(\omega_M))$ and $L^*$ maps $L^2(\partial D)$ into $\ker (\A_0^*(\omega_M))$. These facts, together with the Riesz representation theorem show that $L = l\langle \cdot, \Phi \rangle \Psi$ for some constant $l$. To compute $l$, we use the generalized argument principle for operator valued functions. The operator $\A_0$ is known to have a discrete spectrum \cite{first}, so we can find a small neighbourhood $V$ of $\omega_M$ that contains no characteristic values other than $\omega_M$. Then we have
\begin{equation*}
1 = \frac{1}{2\pi i}\text{tr}\int_{\partial V} \A_0(\omega)^{-1}\frac{d}{d\omega}\A_0(\omega)\dx \omega,
\end{equation*}
This, together with Proposition \ref{prop:polepencil} gives
\begin{align*}
1 &= \frac{1}{2\pi i}\text{tr}\int_{\partial V} \frac{L\frac{d}{d\omega}\A_0}{\omega-\omega_M} \dx \omega \\
 &= \frac{l}{2\pi i}\int_{\partial V} \frac{\langle \frac{d}{d\omega}\A_0\Psi,\Phi\rangle}{\omega-\omega_M} \dx \omega,
\end{align*}
and using Cauchy's integral formula we obtain that $l = \frac{1}{\langle \frac{d}{d\omega}\A_0(\omega_M)\Psi,\Phi\rangle}$, so 
\begin{equation*}
L = \frac{\langle \cdot,\Phi\rangle}{\langle \frac{d}{d\omega}\A_0(\omega_M)\Psi,\Phi\rangle}.
\end{equation*}

\subsection{Computation of resonance perturbation}
Again, let $V$ be a neighbourhood  of $\omega_M$, this time containing  only one characteristic value of $\A$, that is, 
$V$ contains only  the characteristic value  $\omega_\epsilon$  corresponding to a perturbation of the characteristic value $\omega_M$ of $\A_0$. Using the eigenvalue perturbation theory found in \cite{Ammari2009_book},  the leading order term of $\omega_\epsilon - \omega_M$ is given by
\begin{equation*}
\omega_\epsilon-\omega_M = -\frac{\epsilon}{2\pi i}\text{tr}\int_{\partial V} \A_0(\omega)^{-1}\A_1(\omega) \dx \omega + \O(\epsilon^2),
\end{equation*}
which gives
\begin{align*}
\omega_\epsilon-\omega_M &= -\frac{\epsilon}{2\pi i}\text{tr}\int_{\partial V} \frac{L\A_1(\omega)}{\omega-\omega_M} \dx \omega + \O(\epsilon^2) \\
&= -\frac{\epsilon}{2\pi i}\int_{\partial V} \frac{l\langle\A_1(\omega)\Psi,\Phi\rangle}{\omega-\omega_M} \dx \omega +\O(\epsilon^2).
\end{align*}
Because $\langle\A_1(\omega)\Phi,\Psi\rangle$ is holomorphic in $\omega$, Cauchy's integral formula yields
\begin{equation} \label{eq:main}
\omega_\epsilon-\omega_M = -\epsilon\frac{\langle\A_1(\omega_M)\Psi,\Phi\rangle}{\langle\frac{d}{d\omega}\A_0(\omega_M)\Psi,\Phi\rangle} + \O(\epsilon^2).
\end{equation}

We now state the main result of this paper, which gives the leading order term in the expansion of the encapsulated bubble resonance frequency.
\begin{thm} \label{thm:main}
	In the quasi-static regime, for any $\epsilon$ sufficiently small, there exists a resonance frequency $\omega_\epsilon=\omega_\epsilon(\epsilon,\delta)$ for the encapsulated bubble such that $\omega_\epsilon(0,\delta) = \omega_M(\delta)$ and 
	\begin{equation} \label{eq:thm}
	\omega_\epsilon =\omega_M + \epsilon\frac{2\pi\omega_Ma\left( \delta_{lw}-1\right) }{4\pi c\left(\gamma_0+\eta_{k_b}c \right) -c^2\left(1-a\right)} + \O(\epsilon\omega^2) + \O(\epsilon^2).
	\end{equation}
\end{thm}

\begin{proof}
We compute the expression \eqnref{eq:main}. We use subscripts to denote a specific component of a vector. Furthermore, as in \eqnref{eq:regime}, we will work in the regime $\omega^2\ln\omega = \O(\delta)$. Using the expressions for $\Psi$ and $\Phi$ from Section \ref{sec:psi}, we find that the numerator in \eqnref{eq:main} is given by
\begin{equation}\label{eq:nom}
\big\langle\A_1(\omega_M)\Psi,\Phi\big\rangle = \big\langle (\A_1(\omega_M)\Psi_0)_1,u_1\big\rangle + \big\langle (\A_1(\omega_M)\Psi_0)_4, \phi_0\big\rangle +\O(\delta^2).
\end{equation}
We begin with the first term of the numerator. Using the low-frequency expansions \eqnref{eq:Sexpansion} and \eqnref{eq:Kexpansion}, we find
\begin{align*}
\left(\A_1(\omega_M,\delta)\Psi\right)_1 &= a\delta_{lw}\left(\frac{1}{2}I + \K_D^*\right)[\psi_0] + \left(a_l-a\delta_{lw}\right) \left(-\frac{1}{2}I + \K_D^*\right)[\psi_0] - a\left(\K_D + \K_D^*\right)[\psi_0] -a\hat{\S}_D^{k_w}[\tau\psi_0] + \O(\delta)\\
&=  a\delta_{lw}\psi_0 - a\left(\frac{1}{2}I + \K_D\right)[\psi_0] -a\hat{\S}_D^{k_w}[\tau\psi_0] + \O(\delta).
\end{align*}
From this, we can compute 
\begin{align*}
\big\langle (\A_1(\omega_M)\Psi_0)_1,u_1\big\rangle &= \bar{y}_1\left( a\delta_{lw}\langle\psi_0,\psi_0\rangle - a\big\langle\left(\frac{1}{2}I + \K_D\right)[\psi_0],\psi_0\big\rangle -a\big\langle\hat{\S}_D^{k_w}[\tau\psi_0],\psi_0\big\rangle \right)  + \O(\delta^2) \\
&= \bar{y}_1\left( a\left(\delta_{lw} - 1\right) -a\big\langle\hat{\S}_D^{k_w}[\tau\psi_0],\psi_0\big\rangle \right) + \O(\delta^2).
\end{align*}
Define $c_\tau$ as $c_\tau = \langle \tau\psi_0,\phi_0 \rangle$. Then we can compute the term  $\big\langle\hat{\S}_D^{k_w}[\tau\psi_0],\psi_0\big\rangle$ as
$$
\langle \hat\S_D^k[\tau\psi_0], \psi_0 \rangle = \langle \tau\psi_0, \left(\hat\S_D^k\right)^*[\psi_0] \rangle = \langle \tau\psi_0, (\gamma_0 + \bar{\eta}_kc)\phi_0 \rangle = \left(\gamma_0 + \eta_kc\right)c_\tau.
$$
Using this expression, we find that 
\begin{align}
\big\langle (\A_1(\omega_M)\Psi_0)_1,u_1\big\rangle &= \bar{y}_1\left( a\left(\delta_{lw} - 1\right) -a\left(\gamma_0 + \eta_kc\right)c_\tau \right) +\O(\delta^2) \nonumber \\
&= -\frac{\delta}{\gamma_0 + \eta_{k_w}c}\left( a\left(\delta_{lw} - 1\right) -a\left(\gamma_0 + \eta_kc\right)c_\tau \right) +\O(\delta^2) \nonumber \\
&=-\delta a\left( \frac{\delta_{lw}-1}{\gamma_0+\eta_{k_w}c}  - c_\tau\right) +\O(\delta^2). \label{eq:first}
\end{align}
We now turn to the second term of the numerator. It is easily shown that $\S^{k}_D[\psi_0] = a\phi_0 + \O(\omega^2\ln \omega)$ for some constant $a$, so $\frac{\partial^2}{\partial T^2}\S^{k}_D[\psi_0] = \O(\omega^2\ln \omega)$. Hence
\begin{align*}
\left(\A_1(\omega_M,\delta)\Psi\right)_4 &= a\delta_{lw}\delta_{bl}\Rc_D^{k_l}[\psi_0] + (a_l - a\delta_{lw})\delta_{bl}\left(\tau I + \Rc_D^{k_l}\right)[\psi_0] - a\delta\K_1^{k_w}[\psi_0] +\O(\delta^2)\\
&= -a\delta \tau\psi_0-a\delta \left(\K_D^*[\tau\psi_0]-\frac{\tau}{2}\psi_0 + \frac{\partial D_D^{k_w}}{\partial \nu}[\psi_0] \right) +\O(\delta^2) \\
&= -a\delta \left( \K_D^*[\tau\psi_0] +\frac{\tau}{2}\psi_0 + \frac{\partial D_D^{k_w}}{\partial \nu}[\psi_0] \right) +\O(\delta^2).
\end{align*}
It follows that 
\begin{align*}
\big\langle (\A_1(\omega_M)\Psi_0)_4, \phi_0\big\rangle &= -a\delta\left( \langle \tau\psi_0,K_D[\phi_0]\rangle + \frac{1}{2}\langle \tau\psi_0,\phi_0\rangle + \big\langle \frac{\partial D_D^{k_w}}{\partial \nu}[\psi_0],\phi_0 \big\rangle\right) +\O(\delta^2) \\
&= -a\delta\left( c_\tau + \big\langle \frac{\partial D_D^{k_w}}{\partial \nu}[\psi_0],\phi_0 \big\rangle \right) +\O(\delta^2).
\end{align*}
Next we compute the term $\big\langle \frac{\partial D_D^{k_w}}{\partial \nu}[\psi_0],\phi_0 \big\rangle.$ We know from the expansion \eqnref{eq:hankel} that $\frac{\partial D_D^{k_w}}{\partial \nu} = \frac{\partial D_D}{\partial \nu} + \O(\omega^2\ln\omega)$. Furthermore, because $\frac{\partial D_D}{\partial \nu}$ is self-adjoint, we have
$$
\big\langle \frac{\partial D_D^{k_w}}{\partial \nu}[\psi_0],\phi_0 \big\rangle = 
\big\langle \psi_0,\frac{\partial D_D}{\partial \nu}[\phi_0] \big\rangle + \O(\delta) = \O(\delta),
$$
where the last step follows from the well-known fact that $\D_D[\phi_0](x) = 1$ for $x\in D$ \cite{MaCMiPaP}. In total, the second term of the numerator is
\begin{align}
\big\langle (\A_1(\omega_M)\Psi_0)_4, \phi_0\big\rangle = -a\delta c_\tau +\O(\delta^2). \label{eq:second}
\end{align}

Next consider the denominator $\langle \frac{d}{d\omega}\mathcal{A}_{0} \Psi,\Phi \rangle$. As with equation \eqnref{eq:nom}, we have
\begin{equation}\label{eq:denom}
\big\langle\frac{d}{d\omega}\A_0(\omega_M)\Psi,\Phi\big\rangle = \big\langle \frac{d}{d\omega}(\A_0(\omega_M)\Psi_0)_1,y_1\big\rangle + \big\langle \frac{d}{d\omega}(\A_0(\omega_M)\Psi_0)_4, \phi_0\big\rangle +\O(\delta^2).
\end{equation}
We begin with the first term of the denominator. Using the asymptotic expansion of the fundamental solution for small $\omega$ given in equation \eqnref{eq:hankel}, one can see that the following approximation holds:
\begin{align*}
\frac{d}{d\omega}(\A_0(\omega_M)\Psi_0)_1 &= \frac{1}{2\pi\omega_M} \left( \int_{\partial D} \psi_0 - a\int_{\partial D} \psi_0 \right)\phi_0  + \O(\omega\ln\omega) \\
&= \frac{c}{2\pi\omega_M}\left(1-a\right)\phi_0 + \O(\omega\ln\omega).
\end{align*}
It follows that 
\begin{align}
\big\langle \frac{d}{d\omega}(\A_0(\omega_M)\Psi_0)_1,y_1\big\rangle &= \bar{y}_1 \frac{c}{2\pi\omega_M}\left(1-a\right)\langle\phi_0,\psi_0\rangle \nonumber \\
&= -\frac{\delta}{\gamma_0 + \eta_{k_w}c}\frac{c^2}{2\pi\omega_M}\left(1-a\right) + \O(\omega^3\ln\omega). \label{eq:third}
\end{align}
To compute the second term of the denominator, we use the expansion \eqnref{eq:Kexpansion} to find that
\begin{align*}
\frac{d}{d\omega}(\A_0(\omega_M)\Psi_0)_4 &= 2\frac{\omega_M\ln\omega_M}{v_b^2} \K_{D,1}^{(1)}[\psi_0] + \O(\omega).
\end{align*}
It follows that 
\begin{align}
\big\langle \frac{d}{d\omega}(\A_0(\omega_M)\Psi_0)_4, \phi_0\big\rangle &= 2\frac{\omega_M\ln\omega_M}{v_b^2}\big\langle \psi_0,\left(\K_{D,1}^{(1)}\right)^*[\phi_0]\big\rangle + \O(\omega) \nonumber \\
&= \frac{8b_1Vol(D) c}{v_b^2}\omega_M\ln\omega_M + \O(\omega), \label{eq:fourth}
\end{align}
where we have used the fact that $\left(\K_{D,1}^{(1)}\right)^*[\phi_0] = 4b_1Vol(D)\phi_0$ \cite{first}. 

In total, combining \eqnref{eq:first}, \eqnref{eq:second}, \eqnref{eq:third} and \eqnref{eq:fourth} we have
\begin{align*}
\omega_\epsilon-\omega_M &= -\epsilon\frac{-\delta a\left( \frac{\delta_{lw}-1}{\gamma_0+\eta_{k_w}c}  - c_\tau\right) - a\delta c_\tau}{\frac{8b_1Vol(D) c}{v_b^2}\omega_M\ln\omega_M -\frac{\delta}{\gamma_0 + \eta_{k_w}c}\frac{c^2}{2\pi\omega_M}\left(1-a\right)} + \O(\epsilon\omega^2) + \O(\epsilon^2) \\
&=\epsilon\frac{2\pi\omega_Ma\left( \delta_{lw}-1\right) }{4\pi c\left(\gamma_0+\eta_{k_b}c \right) -c^2\left(1-a\right)} + \O(\epsilon\omega^2) + \O(\epsilon^2),
\end{align*}
which proves the theorem.

\end{proof}

\begin{rmk}
From formula \eqnref{eq:thm}, we see that, to leading order, $\omega_\epsilon = \omega_M$ if $\delta_{lw} = 1$, \ie{} there is no shift in the resonance if there is no density contrast between the layer of coating and the water. This is expected in the case when $\kappa_l = \kappa_w$, \ie{} if there is no contrast in the bulk modulus. In this case the layer of coating and the water have identical wave properties, but the formula shows that this is true even if we have a contrast in the bulk modulus. 
\end{rmk}

\begin{rmk}
	All the terms in equation \eqnref{eq:thm} can be numerically computed using standard methods. Moreover, the equation simplifies when $D$ is a circle with radius $r$. In this case we have
	$$
	\gamma_0 = \frac{\ln(r)}{2\sqrt{\pi}} , \quad a = \frac{\eta_{k_b}}{\eta_{k_w}}, \quad	c = \sqrt{2\pi r}.
	$$
	For the unit circle with $r=1$, we get 
	$$
	\omega_\epsilon-\omega_M =\epsilon\frac{\omega_M\eta_{k_b}\left( \delta_{lw}-1\right) }{4\pi \eta_{k_b}\eta_{k_w}  -\left(\eta_{k_w}-\eta_{k_b}\right)} + \O(\epsilon\omega^2) + \O(\epsilon^2). 
	$$
\end{rmk}

\section{Numerical illustration} \label{sec-numerics}
Here we give numerical examples to verify the formula \eqnref{eq:thm} for the encapsulated bubble frequency in the specific case of a circular bubble. In this case, the resonance frequency is easily computed using the multipole method. 

Consider again the equation \eqnref{eq:scattering}. If $D$ is a circle with radius $R$, the encapsulated bubble $D_d$ will also be a circle with radius $R+\epsilon$. Using polar coordinates $(r,\theta)$, is clear that the solution $u$ can be written as
\begin{equation*}
u(x) = \begin{cases}
\sum_{n=-\infty}^\infty a_n J_n(k_br)e^{in\theta} \quad &\text{if } r<R ,\\
\sum_{n=-\infty}^\infty \big(b_n J_n(k_lr) + c_n H_n(k_lr)\big) e^{in\theta}  &\text{if } R<r<R+\epsilon. \\
\sum_{n=-\infty}^\infty d_n H_n(k_wr) e^{in\theta}  &\text{if } R+\epsilon < r,
\end{cases}
\end{equation*}
for some set of constants $a_n, b_n, c_n, d_n, \ n\in \Z$. Here, $J_n$ is the Bessel function of the first kind and of order $n$. Using the boundary conditions, we find that the constants satisfy
\begin{equation*} 
\begin{pmatrix}
J_n(k_bR) &  -J_n(k_lR) & -H_n(k_lR) & 0 \\
0 & J_n(k_l(R+\epsilon)) & H_n(k_l(R+\epsilon)) & -H_n(k_w(R+\epsilon)) \\
k_bJ_n'(k_bR) & -\delta_{bl} kJ_n'(k_lR) & -\delta_{bl} k H'_n(k_lR) & 0 \\
0 & k_lJ_n'(k_l(R+\epsilon)) & k_lH'_n(k_l(R+\epsilon)) & -\delta_{lw}k_wH'_n(k_w(R+\epsilon))  \\
\end{pmatrix}
\begin{pmatrix}
a_n \\ b_n \\ c_n \\ d_n
\end{pmatrix} = 0,
\end{equation*}
for all $n\in \Z$. We seek  $\omega$ such that for some $n$, the corresponding
system is not invertible. In particular, we seek the
encapsulated bubble resonance, which corresponds
to the lowest resonance of the system. It is clear that at the lowest resonant frequency this system features a 
 factor with $n=0$, because the lowest resonance has the least number of oscillations. Thus, at the lowest resonant frequency the matrix 
\begin{equation} \label{eq:multipole}
A(\omega) = \begin{pmatrix}
J_0(k_bR) &  -J_0(k_lR) & -H_0(k_lR) & 0 \\
0 & J_0(k_l(R+\epsilon)) & H_0(k_l(R+\epsilon)) & -H_0(k_w(R+\epsilon)) \\
k_bJ_0'(k_bR) & -\delta_{bl} k_lJ_0'(k_lR) & -\delta_{bl} k_lH'_0(k_lR) & 0 \\
0 & k_lJ_0'(k_l(R+\epsilon)) & k_lH'_0(k_l(R+\epsilon)) & -\delta_{lw}k_wH'_0(k_w(R+\epsilon))  \\
\end{pmatrix}
\end{equation}
becomes singular. Approaching this as a root-finding problem and setting $f(\omega) = \det(A(\omega))$,  we have 
\begin{equation*}
\hat{\omega}_\epsilon = \min_{\omega \in \C} \{ \omega \mid  f(\omega) = 0 \}.
\end{equation*}
The equation $f(\omega) = 0$ can be solved numerically using Muller's method \cite{Ammari2009_book}, and by choosing initial values 
in the vicinity of the Minnaert resonance of the uncoated bubble we can ensure that we find the lowest resonance frequency.

The numerically computed resonance  $\hat{\omega}_\epsilon$  is compared against the encapsulated bubble resonance
${\omega}_\epsilon$ given by equation \eqnref{eq:thm} in Figure \ref{fig:solPos}, for the case $v_w = v_l = v_b = 1$, radius $R=0.5$, $\delta =10^{-3}$ and $\delta_{lw} = 0.5$. In this case, the coating shifts the frequency to a higher value. 
We observe a good agreement between the numerical resonance $\hat{\omega}_\epsilon$, computed using  \eqnref{eq:multipole}, and
the ``Gohberg-Sigal'' resonance ${\omega}_\epsilon$, computed using \eqnref{eq:thm}, as the thickness of the coating decreases.
In Figure \ref{fig:solNeg}, we again plot the encapsulated bubble resonance $\hat{\omega}_\epsilon$,  using the same parameters as in Figure \ref{fig:solPos} apart from the contrast parameter $\delta_{lw}$  which this time we set to be $\delta_{lw} = 1.5$ instead. It can be seen that the presence of the coating shifts the frequency to a lower value. In summary, if $\delta_{lw} > 1$, the layer of coating increases
 the effective density contrast between the gas and the liquid, resulting in a lower resonance frequency, while  on the other hand, if $\delta_{lw} < 1$,  the effective density contrast is reduced which leads to a higher resonance frequency. 

\begin{figure}[H]
	\centering
 	\includegraphics[scale=0.45]{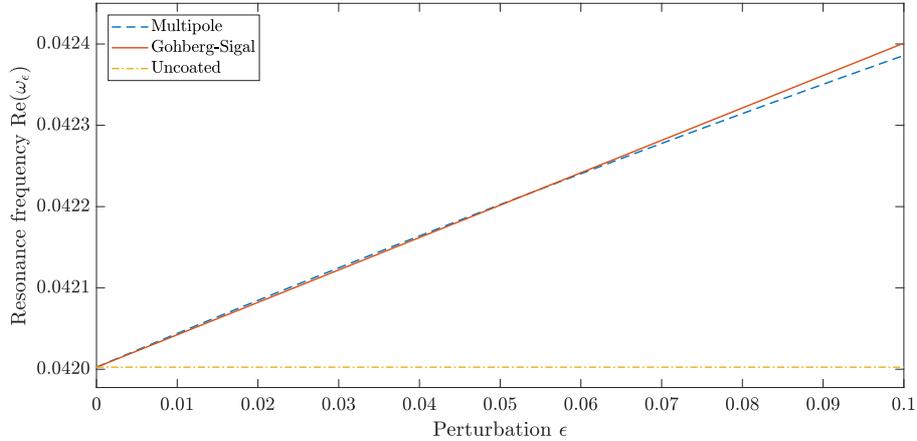}
	\caption{Comparison between the numerically
computed resonance $\hat{\omega}_\epsilon$ and the resonance ${\omega}_\epsilon$ given
by the formula in \eqnref{eq:thm}, for the case of a disk with
radius $R= 0.5$, $\delta = 10^{-3}$ and $\delta_{lw} = 0.5$.}
	\label{fig:solPos}
\end{figure}

\begin{figure}[H]
	\centering
	\includegraphics[scale=0.45]{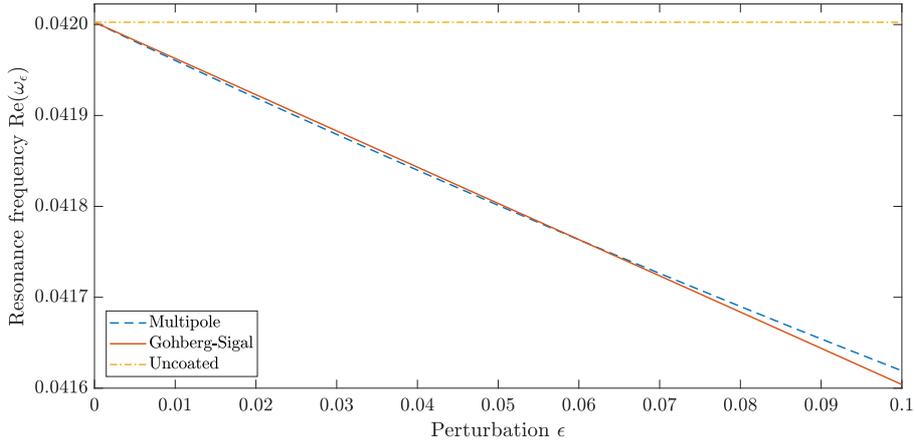}
	\caption{Comparison between the numerically
computed resonance $\hat{\omega}_\epsilon$ and the resonance ${\omega}_\epsilon$ given
by the formula in \eqnref{eq:thm}, for the case of a disk with
radius $R= 0.5$, $\delta = 10^{-3}$ and $\delta_{lw} = 1.5$.}
	\label{fig:solNeg}
\end{figure}

In Figure \ref{fig:err},  we plot the relative error of the $\omega_\epsilon$ when $\epsilon$ is fixed and $\delta \in [10^{-8}, 10^{-2}]$. As expected, the error reduces with decreasing $\delta$, and for $\delta = 10^{-3}$ (approximately the value for water and air) the error is $\approx 0.04\%$.

\begin{figure}[H]
	\centering
	\includegraphics[scale=0.45]{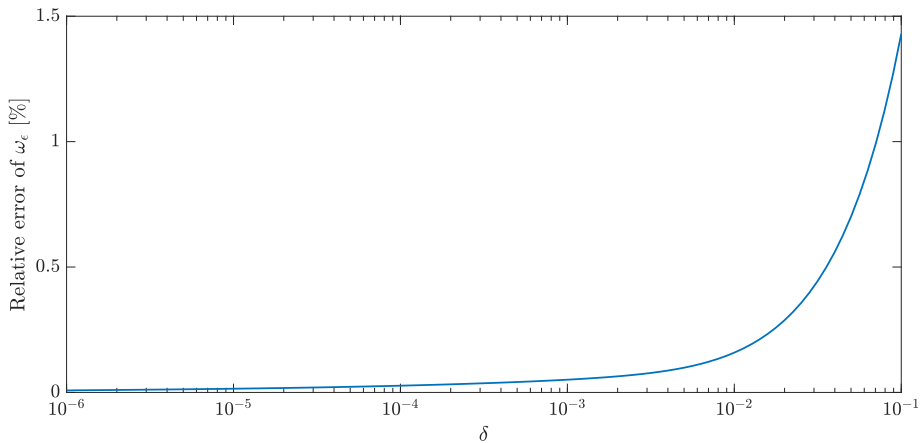}
	\caption{Relative error of $\omega_\epsilon$ as a function of $\delta$, for the case of a disk with radius $R= 0.5$, $\delta = 10^{-3}$, $\delta_{lw} = 0.5$ and fixed $\epsilon = 0.1R$. }
	\label{fig:err}
\end{figure}

\section{Concluding remarks} \label{sec-remarks}
In this paper,  we have proved an original asymptotic formula for the resonance shift that occurs
when a gas bubble in water is encased in a thin layer of coating. The formula is valid for an arbitrarily shaped bubble, and we have numerically verified it in the case of a circular bubble. The findings are of interest for the application of encapsulated bubbles as ultrasound contrast agents. Furthermore, the subwavelength nature of the encapsulated bubble resonance implies that the encapsulation of bubbles can be a useful approach when synthesizing bubbly phononic crystals. In future work, we plan to study wave scattering
by encapsulated bubbles in the full elastic case, thereby providing an even more realistic description of encapsulated bubbles as ultrasound contrast agents.

\bibliography{coating}{}

\begin{thebibliography}{10}

\bibitem{metasurface}
H.~{Ammari}, B.~{Fitzpatrick}, D.~{Gontier}, H.~{Lee}, and H.~{Zhang}.
\newblock {A mathematical and numerical framework for bubble meta-screens}.
\newblock {\em SIAM J. Appl. Math.}, 77(5):1827--1850, 2017.

\bibitem{superfocusing}
H.~{Ammari}, B.~{Fitzpatrick}, D.~{Gontier}, H.~{Lee}, and H.~{Zhang}.
\newblock {Sub-wavelength focusing of acoustic waves in bubbly media}.
\newblock {\em Proc. Royal Soc. A}, 473:20170469, 2017.

\bibitem{first}
H.~Ammari, B.~Fitzpatrick, D.~Gontier, H.~Lee, and H.~Zhang.
\newblock Minnaert resonances for acoustic waves in bubbly media.
\newblock {\em Annales de l'Institut Henri Poincaré C, Analyse non linéaire},
  2018.

\bibitem{MaCMiPaP}
H.~Ammari, B.~Fitzpatrick, H.~Kang, M.~Ruiz, S.~Yu, and H.~Zhang.
\newblock {\em Mathematical and Computational Methods in Photonics and
  Phononics}, volume 235 of {\em Mathematical Surveys and Monographs}.
\newblock American Mathematical Society, 2018.

\bibitem{bandgap}
H.~{Ammari}, B.~{Fitzpatrick}, H.~{Lee}, S.~{Yu}, and H.~{Zhang}.
\newblock {Subwavelength phononic bandgap opening in bubbly media}.
\newblock {\em J. Differential Equations}, 263(9):5610--5629, 2017.

\bibitem{defect}
H.~{Ammari}, B.~{Fitzpatrick}, E.~{Orvehed Hiltunen}, and S.~{Yu}.
\newblock {Subwavelength localized modes for acoustic waves in bubbly crystals
  with a defect}.
\newblock {\em SIAM Journal on Applied Mathematics: to appear}, Apr. 2018.

\bibitem{asymptotics}
H.~Ammari, H.~Kang, and E.~Kim.
\newblock Approximate boundary conditions for patch antennas mounted on thin
  dielectric layers.
\newblock {\em Commun. Comput.Phys.}, 1:1076--1095, 2006.

\bibitem{Ammari2009_book}
H.~Ammari, H.~Kang, and H.~Lee.
\newblock {\em Layer potential techniques in spectral analysis}, volume 153.
\newblock American Mathematical Society Providence, 2009.

\bibitem{effectivemedium}
H.~Ammari and H.~Zhang.
\newblock Effective medium theory for acoustic waves in bubbly fluids near
  minnaert resonant frequency.
\newblock {\em SIAM J. Math. Anal.}, 49(4):3252--3276, 2017.

\bibitem{reviewDoinikov}
A.~A. Doinikov and A.~Bouakaz.
\newblock Review of shell models for contrast agent microbubbles.
\newblock {\em IEEE Transactions on Ultrasonics, Ferroelectrics, and Frequency
  Control}, 58(5):981--993, May 2011.

\bibitem{ultrafastultrasound}
C.~Errico, J.~Pierre, S.~Pezet, Y.~Desailly, Z.~Lenkei, O.~Couture, and
  M.~Tanter.
\newblock Ultrafast ultrasound localization microscopy for deep
  super-resolution vascular imaging.
\newblock {\em Nature}, 527:499 EP --, Nov 2015.

\bibitem{reviewFaez}
T.~Faez, M.~Emmer, K.~Kooiman, M.~Versluis, A.~F.~W. van~der Steen, and
  N.~de~Jong.
\newblock 20 years of ultrasound contrast agent modeling.
\newblock {\em IEEE Transactions on Ultrasonics, Ferroelectrics, and Frequency
  Control}, 60(1):7--20, January 2013.

\bibitem{Gohberg1971}
I.~Gohberg and E.~Sigal.
\newblock An operator generalization of the logarithmic residue theorem and the
  theorem of {R}ouch\'{e}.
\newblock {\em Sb. Math.}, 13(4):603--625, 1971.

\bibitem{asymptoticsderivative}
A.~Khelifi and H.~Zribi.
\newblock Asymptotic expansions for the voltage potentials with two-dimensional
  and three-dimensional thin interfaces.
\newblock {\em Mathematical Methods in the Applied Sciences},
  34(18):2274--2290, 2011.

\bibitem{DesignOfBubbles}
V.~Leroy, A.~Bretagne, M.~Fink, H.~Willaime, P.~Tabeling, and A.~Tourin.
\newblock Design and characterization of bubble phononic crystals.
\newblock {\em Applied Physics Letters}, 95(17):171904, 2009.

\bibitem{DesignOfBandgap}
V.~Leroy, A.~Bretagne, M.~Lanoy, and A.~Tourin.
\newblock Band gaps in bubble phononic crystals.
\newblock {\em AIP Advances}, 6(12):121604, 2016.

\bibitem{Minnaert}
M.~Minnaert.
\newblock {{O}n musical air-bubbles and the sounds of running water}.
\newblock {\em The London, Edinburgh, Dublin Philos. Mag. and J. of Sci.},
  16:235--248, 1933.

\bibitem{weakdifffcn}
W.~P. Ziemer.
\newblock {\em Weakly Differentiable Functions}.
\newblock Springer-Verlag New York, Inc., New York, NY, USA, 1989.

\end{thebibliography}
\end{document}